\pdfoutput=1
\documentclass[11pt,oneside]{amsart}
\usepackage[
paper=a4paper,
headsep=20pt,text={136mm,208mm},centering,includehead
]{geometry}
\usepackage{graphicx}
\usepackage[dvipsnames,usenames]{xcolor}
\usepackage[colorlinks=true,urlcolor=ForestGreen,linkcolor=ForestGreen,citecolor=ForestGreen]{hyperref}
\hypersetup{
	linkbordercolor={1 0 0}, % set to white
	citebordercolor={0 1 0} % set to white
}
\usepackage{cite}
\usepackage{wasysym}
\usepackage{amssymb,mathtools,mathrsfs,stmaryrd,mwe}
\usepackage{bbm}
\usepackage[UKenglish]{babel}
\usepackage[noabbrev]{cleveref}
\usepackage{tikz}
\usetikzlibrary{trees,decorations.pathmorphing,decorations.markings,matrix,shapes}
\usepackage[backgroundcolor=green,linecolor=green,bordercolor=green,disable]{todonotes}
\usepackage{enumerate}
\usepackage{csvsimple}
\usepackage{caption} 
\usepackage{libertine}
\usepackage[vvarbb,bigdelims,cmintegrals,libertine]{newtxmath}

\iftrue
\makeatletter
\def\@settitle{%
	\vspace*{10pt}
	\begin{flushleft}%
		\LARGE\bfseries
		\strut\@title\strut
	\end{flushleft}%
}
\def\@setauthors{%
	\begingroup
	\def\thanks{\protect\thanks@warning}%
	\trivlist
	\raggedright
	\large \@topsep28\p@\relax
	\advance\@topsep by -\baselineskip
	\item\relax
	\author@andify\authors
	\def\\{\protect\linebreak}%
	\authors
	\ifx\@empty\contribs
	\else
	,\penalty-3 \space \@setcontribs
	\@closetoccontribs
	\fi
	\normalfont
	\endtrivlist
	\endgroup
}
\def\@setaddresses{\par
	\nobreak \begingroup
	\small\raggedright
	\def\author##1{\nobreak\addvspace\smallskipamount}%
	\def\\{\unskip, \ignorespaces}%
	\interlinepenalty\@M
	\def\address##1##2{\begingroup
		Address:
		\@ifnotempty{##1}{(\ignorespaces##1\unskip) }%
		{\ignorespaces##2}\par\endgroup}%
	\def\curraddr##1##2{\begingroup
		\@ifnotempty{##2}{\nobreak\noindent\curraddrname
			\@ifnotempty{##1}{, \ignorespaces##1\unskip}\/:\space
			##2\par}\endgroup}%
	\def\email##1##2{\begingroup
		\@ifnotempty{##2}{\nobreak\noindent E-mail address%
			\@ifnotempty{##1}{\par\ignorespaces##1\unskip}\/:\space
			\ttfamily##2\par}\endgroup}%
	\def\urladdr##1##2{\begingroup
		\def~{\char`\~}%
		\@ifnotempty{##2}{\nobreak\noindent\urladdrname
			\@ifnotempty{##1}{, \ignorespaces##1\unskip}\/:\space
			\ttfamily##2\par}\endgroup}%
	\addresses
	\endgroup
	\global\let\addresses=\@empty
}
\def\@setabstracta{%
	\ifvoid\abstractbox
	\else
	\skip@20pt \advance\skip@-\lastskip
	\advance\skip@-\baselineskip \vskip\skip@
	\box\abstractbox
	\prevdepth\z@ % because \abstractbox is a \vtop
	\vskip-22pt
	\fi
}
\renewenvironment{abstract}{%
	\ifx\maketitle\relax
	\ClassWarning{\@classname}{Abstract should precede
		\protect\maketitle\space in AMS document classes; reported}%
	\fi
	\global\setbox\abstractbox=\vtop \bgroup
	\normalfont\small
	\list{}{\labelwidth\z@
		\leftmargin0pc \rightmargin\leftmargin
		\listparindent\normalparindent \itemindent\z@
		\parsep\z@ \@plus\p@
		
	}%
	\item[\hskip\labelsep\bfseries\abstractname.]%
}{%
	\endlist\egroup
	\ifx\@setabstract\relax \@setabstracta \fi
}

\def\ps@headings{\ps@empty
	\def\@evenhead{%
		\setTrue{runhead}%
		\normalfont\scriptsize
		\rlap{\thepage}\hfill
		\def\thanks{\protect\thanks@warning}%
		\leftmark{}{}}%
	\def\@oddhead{%
		\setTrue{runhead}%
		\normalfont\scriptsize
		\def\thanks{\protect\thanks@warning}%
		\rightmark{}{}\hfill \llap{\thepage}}%
	\let\@mkboth\markboth
}\ps@headings

\def\section{\@startsection{section}{1}%
	\z@{-1.4\linespacing\@plus-.5\linespacing}{.8\linespacing}%
	{\normalfont\bfseries\Large}}
\def\subsection{\@startsection{subsection}{2}%
	\z@{-.8\linespacing\@plus-.3\linespacing}{.5\linespacing\@plus.2\linespacing}%
	{\normalfont\bfseries\large}}
\def\subsubsection{\@startsection{subsubsection}{3}%
	\z@{.7\linespacing\@plus.2\linespacing}{-1.5ex}%
	{\normalfont\bfseries}}
\def\@secnumfont{\bfseries}

\renewcommand\contentsnamefont{\bfseries\large}
\def\@starttoc#1#2{\begingroup
	\setTrue{#1}%
	\par\removelastskip\vskip\z@skip
	\@startsection{}\@M\z@{\linespacing\@plus\linespacing}%
	{.5\linespacing}{%\centering
		\contentsnamefont}{#2}%
	\ifx\contentsname#2%
	\else \addcontentsline{toc}{section}{#2}\fi
	\makeatletter
	\@input{\jobname.#1}%
	\if@filesw
	\@xp\newwrite\csname tf@#1\endcsname
	\immediate\@xp\openout\csname tf@#1\endcsname \jobname.#1\relax
	\fi
	\global\@nobreakfalse \endgroup
	\addvspace{32\p@\@plus14\p@}%
	\let\tableofcontents\relax
}
\def\contentsname{Contents}
\def\l@section{\@tocline{1}{.5ex}{0mm}{5pc}{}}
\def\l@subsection{\@tocline{2}{0pt}{2em}{5pc}{}}
\makeatother

\fi 

\makeatother

\newtheorem{theorem}{Theorem}
\newtheorem*{theorem*}{Theorem}
\newtheorem*{corollary*}{}
\newtheorem{proposition}[theorem]{Proposition}
\newtheorem{corollary}[theorem]{Corollary}

\newtheorem*{conjecture*}{Conjecture}

\newtheorem{question}{Question}

\theoremstyle{definition}
\newtheorem{defin}[theorem]{Definition}
\newenvironment{definition}
{\pushQED{\qed}\defin}
{\popQED\enddefin}

\makeatletter

\providecommand{\proofname}{Proof}
\makeatother

\begin{document}
	
\vspace*{-40pt}
\title{On knots that divide ribbon knotted surfaces}
	
\author{Hans U.\ Boden}
\email{\href{mailto:boden@mcmaster.ca}{boden@mcmaster.ca}}
\address{Mathematics \& Statistics, McMaster University, Hamilton, Ontario, Canada}

\author{Ceyhun Elmacioglu}
\email{\href{mailto:elmacioglu.ceyhun@gmail.com}{elmacioglu.ceyhun@gmail.com}}
%\address{Department of Mathematics, Stony Brook University, Stony Brook, NY, United States}

\author{Anshul Guha}
\email{\href{mailto:anshul.guha@yale.edu}{anshul.guha@yale.edu}}
\address{Department of Mathematics, Yale University, New Haven, CT, United States}

\author{Homayun Karimi}
\email{\href{mailto:homayun.karimi@gmail.com}{homayun.karimi@gmail.com}}
\address{Mathematics \& Statistics, McMaster University, Hamilton, Ontario, Canada}

\author{William Rushworth}
\email{\href{mailto:william.rushworth@ncl.ac.uk}{william.rushworth@ncl.ac.uk}}
\address{School of Mathematics, Statistics and Physics, Newcastle University, United Kingdom}

\author{Yun-chi Tang}
\email{\href{mailto:yunchi.tang@mail.utoronto.ca}{yunchi.tang@mail.utoronto.ca}}
\address{Department of Mathematics, University of Toronto, Ontario, Canada}

\author{Bryan Wang Peng Jun}
\email{\href{mailto:bryanwangpengjun@hotmail.com}{bryanwangpengjun@hotmail.com}}
\address{Department of Mathematics, National University of Singapore, Singapore}
	
\def\subjclassname{\textup{2020} Mathematics Subject Classification}
\expandafter\let\csname subjclassname@1991\endcsname=\subjclassname
\expandafter\let\csname subjclassname@2000\endcsname=\subjclassname
\subjclass{57K10, 57K45}
	
\keywords{2-knot, knotted surface, slice, ribbon, doubly slice}
	
\begin{abstract}
We define a knot to be \emph{half ribbon} if it is the cross-section of a ribbon \(2\)-knot, and observe that ribbon implies half ribbon implies slice. We introduce the \emph{half ribbon genus} of a knot \( K \), the minimum genus of a ribbon knotted surface of which \( K \) is a cross-section. We compute this genus for all prime knots up to \( 12 \) crossings, and many \(13\)-crossing knots. The same approach yields new computations of the double slice genus. We also introduce the \emph{half fusion number} of a knot \( K \), that measures the complexity of ribbon \(2\)-knots of which \( K \) is a cross-section. We show that it is bounded below by the Levine-Tristram signatures, and differs from the standard fusion number by an arbitrarily large amount.
\end{abstract}
	
\maketitle

\section{Introduction}\label{Sec:intro}
Knots in \( S^3 \) naturally appear as equatorial cross-sections of knotted surfaces in \( S^4 \). In this paper we restrict to \emph{ribbon} knotted surfaces, those that are particularly simple Morse-theoretically (see \Cref{Def:ribbonsknot}). We define a knot \( K \) to be \emph{half ribbon} if it is the cross-section of a ribbon knotted \(2\)-sphere in \( S^4 \). Ribbon knots are half ribbon (see \Cref{Prop:rihr}), but the converse is an open question. Half ribbon knots are slice, but the converse is also an open question (as not every knotted \(2\)-sphere is ribbon \cite{Yajima1964}).

Answering these questions would resolve the slice-ribbon conjecture, that posits that every slice knot is ribbon \cite{Fox62}. Despite much effort, and results in both directions, it remains open (see, for example, \cite{Lisca2007,Greene2011,Lecuona2012,Gompf2010,Abe2016}). The notion of half ribbon arises naturally by splitting the slice-ribbon conjecture into two questions: (i) if \( K \) is a cross-section of a \(2\)-knot is it the cross-section of a ribbon \(2\)-knot? and (ii) if \( K \) is a cross-section of a ribbon \(2\)-knot does it possess a ribbon disc?

We also introduce the \emph{half ribbon genus}, \( g_{hr} ( K ) \), of a knot \( K \): the minimum genus of a ribbon knotted surface of which \( K \) is a cross-section. The half ribbon genus is an intermediate between the slice genus, \( g_4 ( K ) \), and double slice genus, \( g_{ds} ( K ) \), in that
\begin{equation}\tag{Prop.~\ref{Prop:bounds}}
	2 g_4 ( K ) \leq g_{hr} ( K ) \leq g_{ds} ( K ).
\end{equation}
It follows that a knot of half ribbon genus one would be a counterexample to the slice-ribbon conjecture. More generally, a knot of odd half ribbon genus would have distinct slice and ribbon genera (see \Cref{Q:srg}).

We determine the half ribbon genus of every prime knot of up to \(12\) crossings to be even. In addition, we calculate \(8\) of the \(65\) previously unknown double slice genera of such knots. The following result is proved in \Cref{Sec:determining}.

\begin{theorem}\label{Thm:comp}
	Let \( K \) be a prime knot with up to \(12\) crossings. Then \( g_{hr} ( K ) = 2 g_4 ( K ) \). If \( K \) is one of the following knots then \( g_{ds} ( K ) = 2 g_4 ( K ) \) also:
	\begin{center}
		\( 9_{37} \), \( 10_{74} \), \( 11n148 \), \( 12a554 \), \( 12a896 \), \( 12a921 \), \( 12a1050 \), \( 12n554 \).
	\end{center}
\end{theorem}

We also calculate the half ribbon genus of \(2156\) \(13\)-crossing knots, and in \(247\) such cases determine the double slice genus.

Orson and Powell constructed a knot \( K_{M,N} \) with \( 2g_{4} ( K_{M,N} ) = M \) and \( g_{ds} ( K_{M,N} ) = N \), for all integers \( 0 \leq M \leq N \) with \(M\) even \cite{Orson2021}. In \Cref{Sec:surfaces} we observe that \( g_{hr} ( K_{M,N} ) = M \), so that the half ribbon and double slice genera differ by an arbitrarily large amount. The analogous question regarding the slice and half ribbon genera is open. Answering it in full generality is at least as hard as resolving the slice-ribbon conjecture (see \Cref{Q:gaps}).

Suppose that \( K \) is a cross-section of a knotted surface \( S \). The calculations of \Cref{Thm:comp} rely on realising band attachments to \( K \) as \(3\)-dimensional \(1\)-handle attachments to \( S \) (see \Cref{Thm:tubes1}). This allows us to prove that if \( J \) is obtained from \( K \) via a sequence of \( \ell \) band attachments then
\begin{equation}\tag{Cor.~\ref{Cor:tubes1}}
	\begin{aligned}
	g_{ds} ( J ) - \ell \leq ~&g_{ds} ( K ) \leq g_{ds} ( J ) + \ell \\
	&g_{hr} ( K ) \leq 2g_{r} ( J ) + \ell
	\end{aligned}
\end{equation}
where \( g_{r} ( J ) \) denotes the ribbon genus. We also prove a version of this result for general ribbon cobordisms, \Cref{Thm:tubes2}, from which results of McDonald \cite[Theorems 3.1, 3.2]{McDonald2019} can be recovered.

We consider a finer notion of complexity for half ribbon knots. Ribbon \(2\)-knots are precisely those obtainable from a disjoint union of trivial \(2\)-spheres by attaching \(1\)-handles. The minimum number of \(1\)-handles required to form a \(2\)-knot \(S\) in this way is the \emph{fusion number}, denoted \( f ( S ) \). We introduce the \emph{half fusion number}, \( f_h ( K ) \), of a half ribbon knot \( K \): it is the minimum \( f ( S ) \) such that \( K \) is a cross-section of \(S\).

As only the trivial \(2\)-knot has fusion number zero it follows that \( f_h (K) = 0 \) if and only if \( g_{ds} ( K ) = 0 \). The Levine-Tristram signatures bound the double slice genus from below \cite{Orson2021}, and we prove that this also holds for the half fusion number:
\begin{equation}
\tag{\Cref{Thm:sigbound}}
	f_{h} ( K ) \geq \underset{{\omega \in S^1 \smallsetminus \lbrace 1\rbrace}}{\text{max}}  \left\lbrace |\sigma_{\omega} (K)| \right\rbrace.
\end{equation}

Together with the observation (made in \Cref{Sec:fusionnumbers}) that \( f ( K ) - g_{ds} ( K ) \) can be arbitrarily large, this poses the question: what is the precise relationship between the double slice genus and half fusion number?

The \emph{fusion number} of ribbon knot \( K \), \( f ( K ) \), is the minimum number of bands in a ribbon disc for \( K \). As described in \Cref{Sec:fusionnumbers} the fusion number is bounded below by the half fusion number. We prove that these quantities can differ by an arbitrarily large amount: for all integers \( 0 \leq M < N \) there exists a ribbon knot \( K \) with
\begin{equation}
\tag{Prop.~\ref{Prop:fgaps}}
f_h ( K ) = M ~\text{and}~ f ( K ) \geq N.
\end{equation}

For such a knot \( K \) arbitrarily more bands are required in a ribbon disc than \(1\)-handles are required to form a ribbon \(2\)-knot (of which \(K\) appears as a cross-section). In other words, the structure of the set of ribbon discs for \(K\) is in some sense distinct to that of such ribbon \(2\)-knots.

\subsubsection*{Conventions} All manifolds and embeddings are smooth and orientable. Knots are labelled as per KnotInfo \cite{knotinfo}.

\section{Dividing ribbon knotted surfaces}\label{Sec:defs}
We recall necessary background in \Cref{Sec:background} before tackling our main objects of study in \Cref{Sec:spheres,Sec:surfaces,Sec:fusionnumbers}.

\subsection{Background}\label{Sec:background}
A \emph{\(1\)-link} is a link in \( S^3 \), and \(1\)-link of one component is a \emph{\(1\)-knot}. A \emph{\(2\)-knot} is an embedding \( S^2 \hookrightarrow S^4 \), and a \emph{surface-knot} is an embedding \( F_i \hookrightarrow S^4 \) for \( F_i \) a closed orientable surface. A surface-knot is \emph{trivial} if it bounds an embedded handlebody in \( S^4 \). All of the above objects are considered up to ambient isotopy.

Henceforth we denote by \( S^3_0 \) an equator in \( S^4 \), and by \( B^4_+ \), \( B^4_- \) the associated hemispheres (so that \( S^4 = B^4_+ \cup_{S^3_0} B^4_- \)).

\begin{definition}\label{Def:divides}
	Let \( K \) be a \(1\)-knot and \( S \) a surface-knot. We say that \emph{\(K\) divides \(S\)} if there exists an equator \( S^3_0 \) such that \( S^3_0 \cap S = K \). We also refer to \( K \) as a \emph{cross-section} of \( S \).
\end{definition}

Every \(1\)-knot divides a surface-knot, and this relationship has been of great interest to low-dimensional topologists for almost a century. Questions on this relationship are broadly of two kinds. First, given a fixed \(1\)-knot how complex are the surface-knots that it divides? Obversely, for surface-knots of a given complexity what are the \(1\)-knots that divide them? The focus of this paper is a question of the second type: studying the \(1\)-knots that divide \emph{ribbon} surface-knots.

\begin{definition}[Ribbon surface-knot]\label{Def:ribbonsknot}
	We say that a surface-knot \( S \) is \emph{ribbon} if it bounds a properly embedded handlebody in \( B^5 \) on which the radial height function restricts to a Morse function without critical points of index \( 2 \).
\end{definition}

Before presenting formal definitions of our main objects of interest we fix some further terminology.

\begin{definition}[Slice, ribbon surface]\label{Def:sliceribbonsurface}
	A \emph{slice surface} for a \(1\)-link \( K \) is a compact orientable surface \( F \) properly embedded in \( B^4 \) such that \( \partial F = K \). A \emph{ribbon surface} is a slice surface on which the radial height function restricts to a Morse function without critical points of index \(2\).
\end{definition}

This Morse theoretic definition of a ribbon surface is equivalent to the definition via surfaces immersed in \( S^3 \) with ribbon singularities (see, for example, \cite[Lemma 11.9]{Kamada2002}). Note that slice surface and surface-knot are distinct concepts, likewise ribbon surface and ribbon surface-knot. 

Given a slice surface \( F \) for a \(1\)-knot \( K \) we may form a surface-knot divided by \( K \) as follows. Regard \( K \) as lying in an equator \( S^3_0 \) and \( F \) in \( B^4_+ \). Denote by \( \overline{F} \) the surface obtained by reflecting \( F \) through \( S^3_0 \). The surface-knot \( F \cup_K \overline{F} \) is known as the \emph{double of \( F \)}, and \( K \) divides it by construction.

An embedded \(2\)-sphere in \( S^4 \) is \emph{standard} if it is in Morse position with exactly two critical points (of index \(0\) and \(2\) necessarily). An isotopy representative of a surface-knot is a \emph{sphere-tube presentation} if it is obtained by attaching \(3\)-dimensional \(1\)-handles to a disjoint union of standard \(2\)-spheres. A surface-knot is ribbon if and only if it possesses a sphere-tube presentation \cite[Section 5.6]{Kamada2017}.

We frequently make use of the fact that the double of a ribbon surface for \( K \) is a ribbon surface-knot divided by \( K \). This is described in \cite[Figure 2]{McDonald2019} and the related discussion (for full details see, for example, \cite[Section 5]{Kamada2017}). We suffice ourselves by observing, as in \Cref{Fig:dbst}, that a ribbon surface is made up of discs and bands; upon doubling discs become trivial \(2\)-spheres and bands become \(1\)-handles in a sphere-tube presentation of the double.

\begin{figure}
	\includegraphics[scale=0.75]{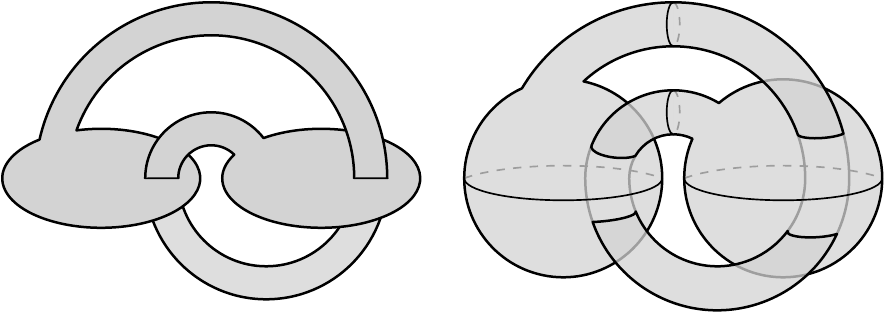}
	\caption{On the left: a ribbon surface \(F\) formed of discs and bands. On the right: the induced sphere-tube presentation of the double of \(F\).}
	\label{Fig:dbst}
\end{figure}

\subsection{Dividing spheres}\label{Sec:spheres}
We say that a \( 1 \)-knot \( K \) is \emph{slice} if it divides a \(2\)-knot \cite{Artin1925,Fox57}, and that \( K \) is \emph{doubly slice} if it divides the trivial \(2\)-knot \cite{Sumners1971,Fox62}.

Using \Cref{Def:ribbonsknot} we introduce a notion that lies between slice and doubly slice.

\begin{definition}[Half ribbon]\label{Def:halfribbon}
	We say that a \(1\)-knot \( K \) is \emph{half ribbon} if it divides a ribbon \(2\)-knot.
\end{definition}

Notice that as the trivial \(2\)-knot is ribbon it follows that a doubly slice \(1\)-knot is half ribbon (the converse fails, for example, on the \(1\)-knot \( 6_1 \)). Half ribbon knots are of course slice, but the status of the converse is an open question (as outlined in \Cref{Q:2sr}).

Recall that a \(1\)-knot \( K \) is \emph{ribbon} if it bounds a ribbon surface of genus \( 0 \); such a surface is known as a \emph{ribbon disc} for \( K \). 

\begin{proposition}\label{Prop:rihr}
	Ribbon \(1\)-knots are half ribbon.
\end{proposition}

\begin{proof}
	Suppose that \( D \) is a ribbon disc for a \(1\)-knot \( K \). The double of \( D \) is a ribbon \(2\)-knot that \( K \) divides by construction.
\end{proof}

Thus doubly slice implies slice but the converse is false, and ribbon implies half ribbon but the converse may be false. In other words, ribbon is to the property of dividing a ribbon \(2\)-knot as doubly slice is to slice, whence the name half ribbon.

\Cref{Prop:rihr} shows that the slice-ribbon conjecture splits into the following questions.

\begin{question}\label{Q:2sr} Let \( K \) be a \(1\)-knot.
	\begin{enumerate}[(i)]
		\item If \( K \) divides a \(2\)-knot must it divide a ribbon \(2\)-knot?
		\item If \( K \) divides a ribbon \( 2 \)-knot must it possess a ribbon disc?
	\end{enumerate}
	A negative answer to (i) or (ii) would yield a counterexample to the slice-ribbon conjecture.
\end{question}

\subsection{Dividing surfaces}\label{Sec:surfaces}
The \emph{slice genus} of a \(1\)-knot \( K \), \( g_4 ( K ) \), is the minimum genus of a slice surface for \( K \). The \emph{ribbon genus} of \( K \), \( g_r ( K ) \), is the minimum genus of a ribbon surface for \( K \).

Notice that taking the double of a slice surface for \( K \) yields a surface-knot divided by \( K \). Generically this surface-knot will be nontrivial. Restricting to trivial surface-knots (not necessarily doubles of slice surfaces for \(1\)-knots) yields the \emph{double slice genus} of \( K \), \( g_{ds} ( K ) \), the minimum genus of a trivial surface-knot divided by \( K \) \cite{Livingston2015}.

Just as in \Cref{Sec:spheres} we can use \Cref{Def:ribbonsknot} to define a quantity intermediate to the slice and double slice genera.

\begin{definition}[Half ribbon genus]\label{Def:hrg}
	Let \( K \) be a \(1\)-knot. The \emph{half ribbon genus of \(K\)}, \( g_{hr} ( K ) \), is the minimum genus of a ribbon surface-knot divided by \( K \). 
\end{definition}

Of course, \( K \) is of half ribbon genus zero if and only if it is half ribbon. The half ribbon genus is finite for all \(1\)-knots as it is bounded above by the double slice genus and twice the ribbon genus.

\begin{proposition}\label{Prop:bounds}
	Let \( K \) be a \(1\)-knot, then
	\begin{equation*}
	2 g_4 ( K ) \leq g_{hr} ( K ) \leq 2g_{r} ( K ) \leq g_{ds} ( K ).
	\end{equation*}
\end{proposition}

\begin{proof}
	Let \( F \) be a ribbon surface for \( K \) with \( g( F) = g_{r} ( K ) \). The double of \( F \) is a ribbon surface-knot of genus \( 2 g_r  ( K ) \) and is divided by \( K \), so that \( g_{hr} ( K ) \leq 2 g_r  ( K )  \). As the double of \( F \) is not necessarily trivial we have \( 2g_{r} ( K ) \leq g_{ds}  ( K ) \).
\end{proof}

\begin{proposition}\label{Prop:subadd}
	The half ribbon genus is subadditive with respect to the connected sum of \( 1 \)-knots.
\end{proposition}

\begin{proof}	
	Let \( K_1 \), \( K_2 \) be \( 1 \)-knots and \( S_1 \), \( S_2 \) ribbon surface-knots such that \( K_i \) divides \( S_i \). Denote by \( S \) the split union of \( S_1 \) and \( S_2 \). That is, \( S \) is a disjoint union of \( S_1 \) and \( S_2 \), and there exists a \( 4 \)-ball \( B \) such that \( S_1 \cap B = S_1 \), \( S_2 \cap B = \varnothing \).
	
	Let \( S' \) be the result of adding a \(1\)-handle between the components of \( S \), chosen so that its core intersects \( \partial B \) in exactly one point and the connected sum of \( K_1 \) and \( K_2 \) appears as the equatorial cross-section.
	
	As \( S_1 \) and \( S_2 \) are ribbon they possess sphere-tube presentations. Let \( S^{\prime\prime} \) be the surface-knot obtained by first isotoping \( S_1 \) and \( S_2 \) into such presentations and then adding a \(1\)-handle between them, the core of which intersects \( \partial B \) in exactly one point. (Notice that the isotopy taking \( S_1 \) and \(S_2\) to their sphere-tube presentations may be chosen to fix \( \partial B \) pointwise.) A schematic for the split union of \( S_1 \) and \(S_2\), together with \( S' \) and \( S^{\prime\prime} \) is given in \Cref{Fig:connectedsums}.
	
	It follows that \( S^{\prime\prime} \) is a ribbon surface-knot as it possesses a sphere-tube presentation. By \cite[Proposition 1.2.11]{Kamada2017} the surface-knots \( S' \) and \( S^{\prime\prime} \) are identical. Thus \( S' \) is a ribbon surface-knot divided by \( K_1 \# K_2 \) of genus \( g(S_1) + g(S_2) \), so that \( g_{hr} ( K_1 \# K_2 ) \leq g_{hr} ( K_1 ) + g_{hr} ( K_2 ) \).
\end{proof}

\begin{figure}
	\includegraphics[scale=0.75]{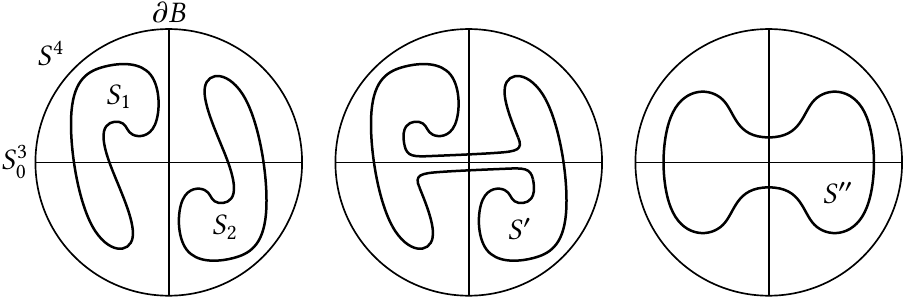}
	\caption{Schematic diagrams of the split union of \( S_1 \) and \(S_2\), and the surface-knots \( S' \) and \( S^{\prime\prime} \).}
	\label{Fig:connectedsums}
\end{figure}

W.\ Chen constructed the first examples of \(1\)-knots of arbitrarily large double slice genus \cite{Chen2021}. These \(1\)-knots are ribbon and are therefore of half ribbon genus zero by \Cref{Prop:rihr}.

Let \( \overline{K} \) denote the mirror image of a \(1 \)-knot \(K\). Orson and Powell showed that the knot \( J = \left( \#^{\frac{M}{2}} \overline{5_2} \right) \# \left( \#^{N-M} 8_{20} \right) \) satisfies \( 2 g_4 ( J ) = M \) and \( g_{ds} ( J ) = N \), for all integers \( 0 \leq M \leq N \) with \( M \) even \cite{Orson2021}. As \( 2 g_4  \left(\overline{5_2}\right)  = g_{ds} \left(\overline{5_2}\right) = 2 \) and \( 8_{20} \) is ribbon we have that \( g_{hr} \left(\overline{5_2}\right) = 2 \) and \( g_{hr} ( 8_{20} ) = 0 \) by \Cref{Prop:bounds}, and \( g_{hr} ( J ) = M \) by \Cref{Prop:subadd}. It follows that the half ribbon and doubly slice genera differ by an arbitrarily large amount. The analogous question regarding the slice genus remains open.

\begin{question}\label{Q:gaps}
	Given integers \( 0 \leq M \leq N \leq P \), does there exist a \(1\)-knot \( K \) such that
	\begin{equation*}
	2 g_{4} ( K ) = M, ~ g_{hr} ( K ) = N, ~ g_{ds} ( K ) = P\text{?}
	\end{equation*}
	Does there exist a prime \(1\)-knot with this property?
\end{question}

Answering \Cref{Q:gaps} for all \( M \), \( N \), \( P \) is at least as hard as finding a \(1\)-knot with distinct slice and ribbon genera.

\begin{question}\label{Q:srg}
	\emph{Does there exist a \(1\)-knot \( J \) of odd half ribbon genus?}
	
	Such a \( J \) must have distinct slice and ribbon genera by \Cref{Prop:bounds}.
\end{question}

In \Cref{Sec:computations} we determine the half ribbon genus of all \(1\)-knots up to \(12\) crossings to be even.

Specialising further, establishing the existence of \(1\)-knots of half ribbon genus one would resolve the slice-ribbon conjecture in the negative.

\begin{question}\label{Q:sr}
	\emph{Does there exist a \(1\)-knot \( K \) of half ribbon genus one?}
	
	Such a \( K \) would be a counterexample to the slice-ribbon conjecture: \( g_4 ( K ) = 0 \) by \Cref{Prop:bounds}, but \( K \) is not ribbon as \( g_{hr} ( K ) \neq 0 \). 
\end{question}

Satoh defined a surjective map from the category of welded knots to that of ribbon tori \cite{Satoh2000}. Can this map be used to address \Cref{Q:sr}?

\subsection{Fusion numbers}\label{Sec:fusionnumbers}
In \Cref{Sec:spheres,Sec:surfaces} we consider the problem of minimising the genus of a surface-knot divided by a given \(1\)-knot. In this section we consider minimising an alternative measure of complexity.

Let \( K \) be a ribbon \( 1 \)-knot. The \emph{fusion number of \( K \)}, \( f ( K ) \), is the minimum number of bands in a ribbon disc for \( K \). Let \( S \) be a ribbon \(2\)-knot. The \emph{fusion number of \( S \)}, \( f ( S ) \), is the minimum number of \(1\)-handles in a sphere-tube presentation for \( S \).

For half ribbon \(1\)-knots we define a new quantity in terms of the fusion number of the ribbon \(2\)-knots they divide.

\begin{definition}[Half fusion number]\label{Def:hf}
	Let \( K \) be a half ribbon \( 1 \)-knot. The \emph{half fusion number of \(K\)}, \( f_{h} ( K ) \), is the minimum fusion number of a ribbon \(2\)-knot divided by \( K \).
\end{definition}

Note that \( f ( S )  = 0 \) if and only if \( S \) is a trivial \(2\)-knot, so that \( f_h ( K ) = 0 \) if and only if \( g_{ds} ( K ) = 0 \). The proof of \Cref{Prop:subadd} also establishes the subadditivity of the half fusion number with respect to the connected sum of \(1\)-knots.

Orson and Powell showed that the Levine-Tristram signatures bound the double slice genus from below \cite{Orson2021}. The same is true of the half fusion number.

\begin{theorem}\label{Thm:sigbound}
	Let \( K \) be a half ribbon \(1\)-knot. Then
	\begin{equation*}
	f_{h} ( K ) \geq \underset{{\omega \in S^1 \smallsetminus \lbrace 1 \rbrace}}{\text{max}}  \left\lbrace |\sigma_{\omega} (K)| \right\rbrace.
	\end{equation*}
\end{theorem}

\begin{proof}
	Let \( S \) be a ribbon \( 2 \)-knot divided by \( K \) such that \( f ( S ) = f_h ( K ) \). By adding \( f ( S ) \) \(1\)-handles to \( S \) it can be converted into a trivial surface-knot \( S' \) \cite{Miyazaki1986}. Moreover, as the result of adding a \(1\)-handle depends only on its core \cite[Proposition 5.1.6]{Kamada2017} these handles may be chosen so that \( K \cup U \) divides \( S' \), for \( K \cup U \) the split union of \( K \) and an unlink \( U \). Generically the handles may intersect the equatorial \( S^3 \) so that we cannot avoid the appearance of this unlink, as per the schematic given in \Cref{Fig:piercinghandles}.
	
	The genus of \( S' \) is equal to \( f ( S ) \) and bounds from above the weak double slice genus of \( K \cup U \), denoted \( g^1_{ds} ( K \cup U ) \) \cite[Equation 1]{Conway2022}. Conway and Orson showed that the Levine-Tristram signatures bound this quantity from below \cite[Corollary 1.3]{Conway2022}, so that
	\begin{equation*}
	\left| \sigma_{\omega} ( K \cup U ) \right| \leq g^1_{ds} ( K \cup U ) \leq g ( S' ) = f_h ( K ).
	\end{equation*}
	The proposition follows by the additivity of the signature under disjoint union.
\end{proof}

\begin{figure}
	\includegraphics[scale=0.75]{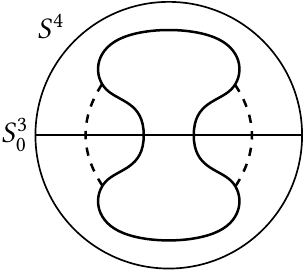}
	\caption{Attaching \(1\)-handles with cores given by the dashed arcs introduces an unlink to the equatorial cross-section.}
	\label{Fig:piercinghandles}
\end{figure}

In the proof above the surface-knot divided by \( K \) is trivialised by attaching \( f_h ( K ) \) \( 1 \)-handles. However, this does not allow us to conclude that \( g_{ds} ( K ) \) bounds \( f_h ( K ) \) from below, as we can guarantee only that \( K \cup U \) divides this trivial surface-knot.

Let \( D \) be a ribbon disc for a \( 1 \)-knot \( K \). The bands of \( D \) yield \(1\)-handles in the double of \( D \), that is a ribbon \(2\)-knot divided by \( K \). It follows that \( f_h ( K ) \leq f ( K ) \). There exist ribbon knots whose fusion and half fusion numbers are arbitrarily far apart.

\begin{proposition}\label{Prop:fgaps}
	For all integers \( 0 \leq M < N \) there exists a ribbon \(1\)-knot \( K \) such that \( f_h ( K ) = M \) and \( f ( K ) \geq N \).
\end{proposition}

\begin{proof}
	Juha\'sz, Miller, and Zemke defined an invariant of \( 1 \)-knots using knot Floer homology, denoted \( \text{Ord}_v ( K ) \), and proved that it bounds the fusion number of ribbon \(1\)-knots from below \cite[Corollary 1.7]{Juhasz2020}. Denote by \( T_{p,q} \) the positive \( (p,q) \)-torus knot and let \( C_{p,q} = T_{p,q} \# \overline{T_{p,q}} \). It is established in \cite[Equation 1.7]{Juhasz2020} that
	\begin{equation*}
	\text{Ord}_v ( C_{p,q} ) = f ( C_{p,q} ) = \text{min} \left\lbrace p, q \right\rbrace - 1.
	\end{equation*}
	
	Notice that \( f_h ( C_{p,q} ) = 0 \) as \( C_{p,q} \) is doubly slice. Let \( K_M = C_{p,q} \# \left( \#^M 8_{20} \right) \). The \(1 \)-knot \( 8_{20} \) is chosen as for \( \omega = e^{\pi i / 3} \) we have \( \sigma_{\omega} ( 8_{20} ) = f_h ( 8_{20} ) = f ( 8_{20} ) = 1 \). Observe that \( M = \sigma_{ \omega } ( K_M ) \leq f_h ( K_M ) \) by \Cref{Thm:sigbound} and the additivity of the signature with respect to connected sum. That the half fusion number is subadditive with respect to connected sum implies that \( f_h ( K_M ) = M  \), in fact.
	
	Further, applying \cite[Equation 1.4]{Juhasz2020} yields
	\begin{equation*}
	\begin{aligned}
	\text{Ord}_v ( K_M ) &= \text{max} \left\lbrace \text{Ord}_v ( C_{p,q} ), \text{Ord}_v ( 8_{20} ) \right\rbrace \\
	&= \text{min} \left\lbrace p, q \right\rbrace - 1
	\end{aligned}
	\end{equation*}
	so that \( \text{min} \left\lbrace p, q \right\rbrace - 1 \leq f ( K_M ) \). A suitably large choice of \( p \) and \( q \) completes the proof.
\end{proof}

\Cref{Prop:fgaps} establishes that the set of ribbon \(2\)-knots divided by a \(1\)-knot \( K \) is in some sense distinct to the set of ribbon discs for \( K \).

Note that as \( g_{ds} ( 8_{20 }) = 1 \) and the double slice genus is subadditive the proof of \Cref{Prop:fgaps} shows that the difference \( f ( K ) - g_{ds} ( K ) \) can also be made arbitrarily large.

\section{Calculations}\label{Sec:computations}
The calculation of the slice genus of prime \(1\)-knots of up to \(12\) crossings has recently been completed, with input from a large number of authors \cite{Lewark2019,Piccirillo2020,Karageorghis2021,Brittenham2021}. Karageorghis and Swenton also calculated the double slice genus of all but \(68\) of these \(1\)-knots \cite{Karageorghis2021}, three of which were later determined by Brittenham and Hermiller \cite{Brittenham2021}.

In this section we calculate the half ribbon genus of every prime \(1\)-knot up to \(12\) crossings. Additionally, we compute the double slice genus in \(8\) of the \(65\) previously undetermined cases. We apply our methods to \(13\)-crossing \(1\)-knots, calculating the half ribbon genus of \(2156\) of them. In \(247\) such cases we are also able to determine the double slice genus.

\Cref{Sec:handles} describes how ribbon cobordisms between \(1\)-knots can be realised on surface-knots they divide, and \Cref{Sec:determining} gives the results of our calculations.

\subsection{Handle attachments defined by ribbon cobordisms}\label{Sec:handles}
Recent calculations of the slice and double slice genera have employed upper bounds obtained from various sequences of band attachments. Our calculations rely on the observation that if \( K \) divides a surface-knot \( S \), attaching bands to \( K \) may be realized by adding \(1\)-handles to \( S \).

\begin{theorem}\label{Thm:tubes1}
	Let \( K \), \( J \) be \(1\)-links, \( C \) a connected cobordism from \( K \) to \( J \) defined by attaching \( \ell \) bands to \( K \), and \( \overline{C} \) its reverse. Suppose that \( J \) divides a surface-knot \( S \) with \( S = S_+ \cup_J S_- \). Then
	\begin{enumerate}[(i)]
		\item The surface-knot \( S' = S_+ \cup C \cup_K \overline{C} \cup S_- \) is obtained from \( S \) by attaching \( \ell \) \(1\)-handles.
		\item If \( S_+ \) is a ribbon surface for \( J \) and \( \overline{S_-} = S_+ \) then \( S' \) is ribbon.
		\item If \( S \) is trivial then \( S' \) is trivial.
	\end{enumerate}
\end{theorem}

\begin{proof}
	\noindent(i): The bands of \( C \) become \(1\)-handles in \( S' \), that may be thought of as being attached to the cylinder \( J \times [0,1] \) in \( S_+ \cup \left( J \times [0,1] \right) \cup S_- \).
	
	\noindent(ii): If \( S_+ \) is a ribbon surface for \( J \) then \( S_+ \cup C \) is a ribbon surface for \( K \), the double of which is a ribbon surface-knot.
	
	\noindent(iii): Attaching \( 1 \)-handles to a trivial surface-knot yields a trivial surface-knot \cite[Proposition 11.2]{Kamada2002}.
\end{proof}

The \emph{superslice genus}, \( g^{s} ( K ) \), of a \(1\)-knot \( K \) is the minimum genus of a slice surface for \( K \) the double of which is a trivial surface-knot \cite{Chen2021}.

\begin{corollary}\label{Cor:tubes1}
	Let \( K \), \( J \) be \(1\)-knots. Suppose that \( J \) is obtained from \( K \) via a sequence of \( \ell \) band attachments. Then
	\begin{enumerate}[(i)]
		\item \( g_{ds} ( J ) - \ell \leq g_{ds} ( K ) \leq g_{ds} ( J ) + \ell \).
		\item \( 2g^{s} ( J ) - \ell \leq 2g^{s} ( K ) \leq 2g^{s} ( J ) + \ell \).
		\item \( g_{hr} ( K ) \leq 2g_{r} ( J ) + \ell \).
	\end{enumerate}
\end{corollary}

\begin{proof}
	\noindent(i): Let \( J \) divide a trivial surface-knot \( S \) with \( g ( S ) = g_{ds} ( J ) \). By \Cref{Thm:tubes1} \( K \) divides a trivial surface-knot of genus \( g_{ds} ( J ) + \ell \), whence the rightmost inequality. Reversing the roles of \( K \) and \( J \) gives the leftmost.
	
	\noindent(ii): Let \( F \) be a slice surface for \( J \) with \( g ( F ) = g^{s} ( J ) \). Denote by \( C \) the cobordism defined by the sequence of \( \ell \) bands. Then \( F \cup C \) is a slice surface for \( K \) of genus \( g^{s} ( J ) + \frac{\ell}{2} \) (the number of bands is even as \( C \) is orientable with two boundary components). The double of \( F \cup C \) is a surface-knot of genus \( 2 g^{s} ( J ) + \ell \), and is trivial by \Cref{Thm:tubes1}. The rightmost inequality follows from the fact that \( K \) divides this surface-knot by construction. Reversing the roles of \( K \) and \( J \) gives the leftmost.
	
	\noindent(iii): Let \( F \) be a ribbon surface for \( J \) with \( g ( F ) = g_{r} ( J ) \). A doubling process similar to that given above shows that \( K \) divides a ribbon surface of genus \( 2 g_r ( K ) + \ell \).
\end{proof}

Picking \( J \) as the unknot in \Cref{Cor:tubes1} (i), (ii) recovers \cite[Theorem 3.1]{McDonald2019}. It is unknown if attaching a \(1\)-handle to a ribbon surface-knot preserves the ribbon property. This causes \Cref{Cor:tubes1} (iii) to be of a different form to \Cref{Cor:tubes1} (i), (ii). The result of attaching a \(1\)-handle \( h \) to a surface-knot \( S \) depends only on the homotopy class of the core, \( \gamma \), of \( h \) in the complement of \( S \) \cite[Proposition 5.1.6]{Kamada2017}. If \( S \) is ribbon it may be isotoped into a sphere-tube presentation. The trace of this isotopy is of codimension \(1\) so that it and \( \gamma \) generically intersect in points. It is therefore unclear if the isotopy can be completed in the presence of \(h\).

Let \(K\) and \(J\) be \(1\)-knots and \(C\) a cobordism between them. We say that \(C\) is a \emph{ribbon cobordism from \(K\) to \(J\)} if we do not encounter a birth of an unknotted and unlinked component when traversing \(C\) from \(K\) to \(J\). A \emph{ribbon concordance} is a ribbon cobordism of genus \( 0 \).

\Cref{Thm:tubes1} is similar to the following description of the union of a ribbon concordance with its reverse given by Zemke \cite{Zemke2019}\footnote{Our definition is consistent with the \emph{reverse} of what Zemke refers to as a ribbon concordance.}. If \( C \) is a ribbon concordance from \( K \) to \( J \) then \( C \cup_K \overline{C} \) is obtained from \( J \times [0,1] \) by taking a disjoint union with trivial \( 2 \)-knots and attaching them to \( J \times [0,1] \) via \(1\)-handles. This operation is known as taking a \emph{tube sum with trivial \(2\)-knots}.

This description can be combined in a straightforward manner with \Cref{Thm:tubes1} to obtain the following results. We expect these more general results to be useful in further study of the double slice and half ribbon genera.

\begin{theorem}\label{Thm:tubes2}
	Let \( K \), \( J \) be \(1\)-links, \( C \) a connected ribbon cobordism from \( K \) to \( J \) with \( s \) saddles and \( d \) deaths. If \( J \) divides a surface-knot \( S \) then \( K \) divides a surface-knot, \( S' \), obtained from \( S \) by attaching \( (s - d) \) \(1\)-handles to \( S \) and taking the tube sum with \( d \) trivial \(2\)-knots.
\end{theorem}

\begin{proof}
	The saddles of \( C \) split into two types. Up to isotopy we may assume that as we move in reverse through \( C \) (from \(J\) to \(K\)) we first see the creation of a \( d \)-component unlink, the components of which are then joined together by bands to produce a \(1\)-knot. The saddles of \( C \) associated to these bands yield the tube sums that contribute to \( S' \). The remaining \( (s - d) \) saddles of \( C \) yield the \(1\)-handles attached to \( S \). A schematic example is given in \Cref{Fig:tubes}.
\end{proof}

\begin{figure}
	\includegraphics[scale=0.75]{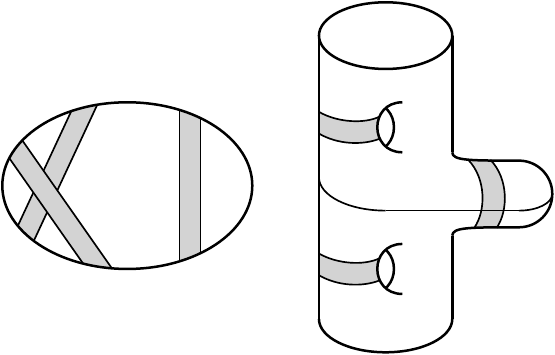}
	\caption{On the left: bands defining a ribbon cobordism \( C \). On the right: a schematic of \( C \cup \overline{C} \) (some handles have been isotoped away from the equator for aesthetic purposes.)}
	\label{Fig:tubes}
\end{figure}

\begin{corollary}\label{Cor:tubes2}
	Suppose that there is a ribbon cobordism \( C \) from \( K \) to \( J \) with \( s \) saddles and \( d \) deaths. Then
	\begin{enumerate}[(i)]
		\item \( g_{ds} ( J ) - s \leq g_{ds} ( K ) \leq g_{ds} ( J ) + s \).
		\item \( g_{hr} ( K ) \leq 2g_{r} ( J ) + s - d \).
	\end{enumerate}
\end{corollary}

\begin{proof}
	\noindent(i): Suppose that \( J \) divides a trivial surface-knot \( S \) with \( g ( S ) = g_{ds} ( J ) \). By \Cref{Thm:tubes2} \( K \) divides a surface-knot, \( S' \), obtained from \( S \) by attaching \( (s - d) \) \(1\)-handles to \( S \) and taking \( d \) tube sums with trivial \(2\)-knots. Denote by \( S_1 \) the result of attaching the \( (s - d) \) \(1\)-handles to \( S \). As \( S \) is trivial \( S_1 \) is trivial, and \( g ( S_1 ) =  g ( S ) + s - d \).
	
	Denote by \( T \) the set of trivial \(2\)-knots that will be tube summed to \( S_1 \) to produce \( S' \). A \( 1 \)-handle is \emph{trivial} if it bounds an embedded \( D^1 \times D^2 \). As the \( d \) components of \( T \) are formed by doubling a ribbon cobordism they may be connected with \( (d - 1) \) trivial \(1\)-handles to produce a single trivial \(2\)-knot, \( T'\), without altering the equatorial cross-section. We may add an additional trivial \(1\)-handle between \( S_1 \) and \( T' \) to produce a trivial surface-knot, \( S_2 \), also without altering the cross-section. Notice that \( g ( S_2 ) = g ( S ) + s - d \).
	
	Finally, consider the set of \( d \) \(1\)-handles defined by the tube sums that produce \( S' \). We may add these \( 1 \)-handles to \( S_2 \) (as it includes both \( S \) and \( T \)) to produce a trivial surface-knot, \( S_3 \), of genus \( g ( S_2 ) + d =  g_{ds} ( J ) + s \). That the trivial \(1\)-handles above are attached without altering the cross-section ensures that \( K \) divides \( S_3 \).
	
	\noindent(ii): Let \( F \) be a ribbon surface for \( J \) with \( g ( F ) = g_{r} ( J ) \). Then \( J \) divides the genus \( 2g_{r} ( J ) \) surface-knot \( S \) with sphere-tube presentation given by the double of \( F \). By \Cref{Thm:tubes2} \( K \) divides a surface-knot \( S ' \) obtained from \( S \) by attaching \(1\)-handles and taking tube sums with trivial \(2\)-knots. Thus \( S' \) is ribbon as it also has a sphere-tube presentation, and \( g ( S ' ) = 2g_{r} ( J ) + s - d \) (only the \( (s - d) \) \(1\)-handles attached to \( S \) affect the genus of \(S'\)).
\end{proof}

Picking \( J \) as the unknot in \Cref{Cor:tubes2} (i) recovers \cite[Theorem 3.2]{McDonald2019}.

\subsection{Determining genera}\label{Sec:determining}
As employed by Lewark-McCoy and Brittenham-Hermiller the operations of switching a crossing, switching a pair of crossings of zero writhe, and taking the oriented resolution at two crossings are all realizable by attaching two bands \cite[Lemma 5]{Lewark2019}.

To calculate the half ribbon genus we combine specific examples of these operations found by Lewark-McCoy and Brittenham-Hermiller, calculations of the double slice genus by Karageorghis-Swenton, and \Cref{Cor:tubes1}.

For the \(1\)-knots that Lewark-McCoy and Brittenham-Hermiller do not provide a suitable operation we made an independent computer search for crossing changes to ribbon or doubly slice \(1\)-knots.

\begin{proof}[Proof of \Cref{Thm:comp}]
The slice-ribbon conjecture has been verified up to \(12\) crossings. Therefore \( 2 g_4 ( K ) = g_{hr} ( K ) = 0 \) for all slice \(1\)-knots.

If \( K \) is not slice and \( g_{ds} ( K ) \) was determined prior to this work then \( 2 g_4 ( K ) = g_{ds} ( K ) \) \cite{knotinfo}, so that \( 2 g_4 ( K ) = g_{hr} ( K ) \) also by \Cref{Prop:bounds}.

This leaves \(58\) cases of undetermined half ribbon genus, as described in \Cref{Tab:genera} (\(7\) of the \(65\) cases undetermined by Karageorghis and Swenton are slice). These \(1\)-knots are obtainable from a ribbon \(1\)-knot by attaching two bands as shown by Lewark-McCoy \cite[Appendix A]{Lewark2019}, Brittenham-Hermiller \cite[Section 4]{Brittenham2021}, or our crossing change search. Therefore these \(1\)-knots have half ribbon genus equal to \(2\) by \Cref{Cor:tubes1} and \Cref{Prop:bounds} (recall that they are not slice). An example of this process in the case of \( 10_{74} \) is given in \Cref{Fig:bounds}.

Finally, as depicted in \Cref{Tab:genera} the \(1\)-knots \( 9_{37} \), \( 10_{74} \), \( 11n148 \), \( 12a554 \), \( 12a896 \), \( 12a921 \), \( 12a1050 \), \( 11n148 \), \( 12n554 \) are obtainable from a doubly slice \(1\)-knot by attaching two bands, so that they have double slice genus \( 2 \) by \Cref{Cor:tubes1} (they were shown to have double slice genus \(2\) or \(3\) by Karageorghis and Swenton).
\end{proof}

\begin{figure}
	\includegraphics[scale=0.75]{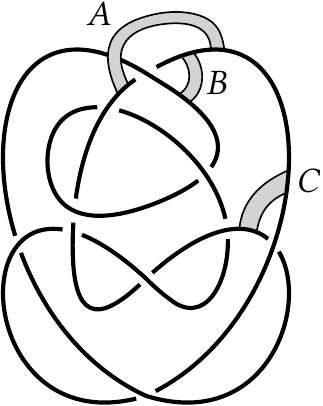}
    \caption{A diagram of the \(1\)-knot \( 10_{74} \), together with three bands. Attaching the bands labelled \(A\) and \(B\) realises a crossing change that yields a diagram of the \(1\)-knot \( 9_{46} \). Subsequently attaching the band labelled \(C\) defines a ribbon disc for \( 9_{46} \).}
    \label{Fig:bounds}
\end{figure}

This leaves \( 57 \) prime \(1\)-knots of up to \(12\) crossings with unknown double slice genus.

Our methods extend fruitfully into \(13\)-crossing \(1\)-knots. Specifically, we restricted to \(13\)-crossing \(1\)-knots of signature \(2\) and searched for crossing changes to ribbon or doubly slice \(1\)-knots. This allows us to show that \(2156\) such \(1\)-knots have half ribbon genus \(2\), of which \(247\) have double slice genus \(2\). The full results of these calculations are provided \href{https://www.mas.ncl.ac.uk/william.rushworth/ccdata.html}{on this webpage} and \href{https://arxiv.org/abs/2209.15577}{on the arXiv}.

Just as this paper studies \(1\)-knots that appear as cross-sections of ribbon \(2\)-knots, one could study those that appear as cross-sections of \emph{homotopy ribbon} \(2\)-knots. In addition to being possibly distinct to half ribbon in the smooth category, such a definition extends to the topological category. The most basic question one might ask in this setting is as follows.

\begin{question}
	Is every topologically slice \(1\)-knot the cross-section of a homotopy-ribbon \(2\)-knot?
\end{question} 

Finally, although we do not pursue it here, the notion of half ribbon genus can readily be extended allow for knotted surfaces of more than one component, as is done for the double slice genus \cite[Equation 1]{Conway2022}. There are natural generalizations of \Cref{Thm:tubes1,Thm:tubes2} to this setting.

\subsubsection*{Acknowledgements} We thank an anonymous referee for their valuable comments. This work was completed during the Fields Undergraduate Summer Research Program 2022. It is our pleasure to thank the Fields Institute for its support and hospitality. KnotInfo was an essential tool during the completion of this work \cite{knotinfo}. H.\ U.\ Boden was partially funded by the Natural Sciences and Engineering Research Council of Canada.

\begin{table}[h]
	\caption{The second and third columns list the result of attaching two bands, as given by \cite[Appendix A]{Lewark2019} and \cite[Section 4]{Brittenham2021}, respectively. The fourth column lists the result of a crossing change, found by our computer search. If we were able to calculate a previously unknown value of \( g_{ds} \) it is listed in the fifth column.}\label{Tab:genera}
	\begin{tabular}{c c}
	{\begin{minipage}{0.49\textwidth}
		%\label{Tab:ghr1}
		\begin{filecontents*}{gds_ghf_1.csv}
knot,lm,bh,cc,gds
{$9_{37}$},,,{$9_{46}$},2
{$9_{48}$},,,{$8_{20}$},
{$10_{74}$},,,{$9_{46}$},2
{$10_{103}$},,,{$8_8$},
11a135,{$6_1$},,,
11a155,{$8_{20}$},,,
11a173,{$8_{20}$},,,
11a327,{$8_{20}$},,,
11a352,{$6_1$},,,
11n71,,,{$8_{20}$},
11n75,,,{$8_{20}$},
11n148,{$4_1 \# 4_1$},,,2
11n167,{$6_1$},,,
12a164,{$8_{20}$},,,
12a166,{$8_{20}$},,,
12a177,{$6_1$},,,
12a247,{$8_8$},,,
12a265,{$6_1$},,,
12a298,{$8_{20}$},,,
12a327,{$8_8$},,,
12a396,{$8_{20}$},,,
12a413,{$8_{20}$},,,
12a449,{$6_1$},,,
12a493,{$6_1$},,,
12a503,{$10_{75}$},,,
12a554,,,{$9_{46}$},2
12a735,{$6_1$},,,
12a750,{$6_1$},,,
12a769,{$6_1$},,,
\end{filecontents*}
\csvreader[head to column names,tabular=|c|c|c|c|c|,table head=\hline Knot & L-M & B-H & C.c. & \( g_{ds} \)  \\\hline,table foot=\hline]{gds_ghf_1.csv}%
		{}%
		{\knot & \lm & \bh & \cc & \gds}
	\end{minipage}}
	&
	{\begin{minipage}{0.49\textwidth}
		%\label{Tab:ghr2}
		\begin{filecontents*}{gds_ghf_2.csv}
knot,lm,bh,cc,gds
12a873,{$8_{20}$},,,
12a895,{$10_{87}$},,,
12a896,{$3_1 \# \overline{3_1}$},,,2
12a905,,{$10_{87}$},,
12a921,{$4_1 \# 4_1$},,,2
12a971,{$6_1$},,,
12a1050,,,{$9_{46}$},2
12a1085,{$6_1$},,,
12a1194,{$8_8$},,,
12a1200,{$6_1$},,,
12a1226,{$8_{20}$},,,
12n147,{$8_8$},,,
12n334,{$6_1$},,,
12n379,{$8_{20}$},,,
12n388,{$6_1$},,,
12n396,{$6_1$},,,
12n460,{$6_1$},,,
12n480,{$6_1$},,,
12n495,{$8_{20}$},,,
12n524,{$6_1$},,,
12n537,{$6_1$},,,
12n554,,,{$9_{46}$},2
12n555,,{$6_1$},,
12n577,{$10_{140}$},,,
12n583,{$6_1$},,,
12n737,{$6_1$},,,
12n813,{$6_1$},,,
12n846,{$6_1$},,,
12n869,{$8_{20}$},,,
\end{filecontents*}
\csvreader[head to column names,tabular=|c|c|c|c|c|,table head=\hline Knot & L-M & B-H & C.c. & \( g_{ds} \)  \\\hline,table foot=\hline]{gds_ghf_2.csv}%
		{}%
		{\knot & \lm & \bh & \cc & \gds}
	\end{minipage}}
\end{tabular}
\end{table}

\bibliographystyle{alpha}
\bibliography{library}
\newpage

\end{document}